\documentclass{amsart}

\usepackage[centertags]{amsmath}
\usepackage{amssymb,amsthm, mathrsfs}
\usepackage[all]{xy}
\usepackage{hyperref}
\usepackage{graphicx}
\usepackage{color}

\title{A cobordism model for Waldhausen $K$-theory}

\author{George Raptis}
\address{\newline
G. Raptis \newline
Fakult\"{a}t f\"ur Mathematik \\
Universit\"{a}t Regensburg \\
D-93040 Regensburg, Germany}
\email{georgios.raptis@ur.de}
\author{Wolfgang Steimle}
\address{\newline
W. Steimle \newline 
Institut f\"ur Mathematik\\
Universit\"at Augburg\\
D-86135 Augsburg, Germany}
\email{wolfgang.steimle@math.uni-augsburg.de}

\DeclareMathOperator{\id}{id}
\newcommand{\set}{\mathrm{Set}}
\newcommand{\sset}{\mathrm{SSet}}

\newcommand{\Cat}{\mathrm{Cat}}

\newcommand{\op}{^{\mathrm{op}}}

\newcommand{\fibr}[2]{\begin{pmatrix} {#1} \\ \downarrow \\ {#2}\end{pmatrix}}

\newcommand{\Rfd}{\mathcal{R}^{hf}}
\newcommand{\F}{\mathsf{F}}

\begin{document}

\theoremstyle{plain}
\newtheorem{thm}{Theorem}
\newtheorem{cor}[thm]{Corollary}
\newtheorem{lem}[thm]{Lemma}
\newtheorem{prop}[thm]{Proposition}
\newtheorem{claim}[thm]{Claim}
\theoremstyle{definition}
\newtheorem{defn}[thm]{Definition}
\newtheorem{obs}[thm]{Observation}
\newtheorem{constr}[thm]{Construction}

\theoremstyle{remark}
\newtheorem{rem}[thm]{Remark}
\newtheorem*{rem*}{Remark}

\newtheorem{exa}[thm]{Example}

\setcounter{secnumdepth}{2}
\numberwithin{thm}{section}

\SelectTips{eu}{10}
\renewcommand{\theenumi}{\roman{enumi}}
\renewcommand{\labelenumi}{\textup{(}\theenumi\textup{)}}
\renewcommand{\theequation}{\arabic{equation}}

\newcommand{\sd}{\operatorname{\mu}}
\newcommand{\Sd}{\operatorname{Sd}}

\newcommand{\Ar}{\operatorname{Ar}}
\newcommand{\TAr}{\widetilde{\Ar}}
\newcommand{\BM}{\mathrm{BM}}
\newcommand{\cob}{\operatorname{\mathsf{Cob}}}
\newcommand{\cobbig}{\cob^{\mathrm{big}}}
\newcommand{\cobpush}{\cob^{\mathrm{po}}}
\newcommand{\cobcof}{\cob^{\mathrm{sym}}}

\newcommand{\calc}{\mathcal C}

\begin{abstract}
We study a categorical construction called the \emph{cobordism category}, which associates to each Waldhausen category a simplicial category of cospans. We prove that this construction 
is homotopy equivalent to Waldhausen's $S_{\bullet}$-construction and therefore it defines a model for Waldhausen $K$-theory. As an example, we discuss this model for $A$-theory and show 
that the cobordism category of homotopy finite spaces has the homotopy type of Waldhausen's $A(*)$. We also review the canonical map from the cobordism category of manifolds to 
$A$-theory from this viewpoint.
\end{abstract}

\maketitle


\section{Introduction} 

Many of the definitions of higher algebraic $K$-theory are fundamentally based on categorical constructions. Some of the main categorical constructions are Quillen's original 
$Q$-construction for exact categories \cite{Qu}, the closely related $s_{\bullet}$-construction, Waldhausen's more general $S_{\bullet}$-construction for Waldhausen categories \cite{Wa}, and 
Thomason's $\mathcal{T}_{\bullet}$-construction. To this incomplete list, we may also add the Gillet-Grayson $G$-construction and Quillen's $S^{-1}S$-construction. Moreover, 
$\infty$-categorical versions of such constructions have also been developed in recent years (see, for exa\-mple, \cite{Ba}).  Where it applies, each of these constructions leads to the 
same (= homotopy equivalent) definition of higher algebraic $K$-theory.  

In this paper we introduce a new such construction for Waldhausen categories. The construction applies, in fact, more generally to unpointed Waldhausen categories where neither an initial nor 
a terminal object is required. This construction is obtained from categories of cospans where one of the arrows is a cofibration. We call this construction the \emph{cobordism category} partly because our inspiration for this construction came from previous work \cite{RS1, RS2} on the relation between the cobordism category of manifolds \cite{GMTW, Ge} with the algebraic $K$-theory of spaces ($A$-theory) \cite{Wa}. Moreover, specifically in this context, the idea to regard a cospan of homotopy finite spaces as a kind of formal cobordism seems particularly illuminating. Conversely, this connection also inspired in \cite{Steimle_add} the use of $K$-theoretic methods in the study of cobordism categories of manifolds.
The main result of the present paper compares this cobordism category construction with the $S_{\bullet}$-construction and shows that the loop space of the classifying space of this cobordism category is homotopy equivalent to Waldhausen $K$-theory.

The cobordism category construction is reminiscent of the $Q$-construction \cite{Qu, Ba}. Each one is related to the $S_{\bullet}$-construction  via Segal's edgewise subdivision 
\cite{Se}, but in ``opposite" ways. The cobordism category can be compared with the $Q$-construction directly by taking (homotopy) pullbacks/pushouts in order to 
exchange a sequence of cospans with a sequence of spans. However, our proof that the cobordism category models Waldhausen $K$-theory does not require a biWaldhausen category structure 
and it applies to all Waldhausen categories. Thus, one could consider our results as a proof that a variation of the $Q$-construction yields the correct homotopy type even in cases which are not stable or additive. In future work, we plan to use the cobordism category model in the comparison 
between cobordism categories of manifolds and the algebraic $K$-theory of spaces, following the work initiated in \cite{RS_hcob}.

In this direction, we also discuss in the present paper models for $A$-theory based on the cobordism category construction. More specifically, we prove in Theorem \ref{model-A-theory2} that the algebraic $K$-theory $A(X)$ of a 
space $X$ can be obtained from the classifying space of a category of formal cobordisms between homotopy finite spaces with a structure map to $X$. Using this cobordism model for $A(X)$, 
we can describe the map from the standard cobordism category of manifolds to Waldhausen's $A$-theory essentially as an inclusion of cobordism categories. 

\medskip

The paper is organized as follows. In Section 2, we define the cobordism category of an unpointed Waldhausen category and discuss some of 
its properties. As the construction uses Segal's edgewise subdivision, we begin with a short review of this subdivision. In a final subsection, we discuss a variant 
of the cobordism category where each cospan consists of ``disjoint'' cofibrations. In Section 3, we 
define a comparison map from the cobordism category to the $S_{\bullet}$-construction. Our main result (Theorem \ref{main}) shows that this induces 
a homotopy equivalence after geometric realization. We also explain that the cobordism category construction can be iterated so that it can also be 
used to obtain the deloopings of Waldhausen $K$-theory.  In Section 4, we discuss cobordism category models for $A$-theory and give a definition of the map 
from the standard cobordism category of manifolds to Waldhausen's $A$-theory using this model.

\medskip

\thanks{\noindent \emph{Acknowledgements.} We warmly acknowledge the support and the hospitality of the Hausdorff Institute for Mathematics in Bonn where 
a preliminary outline of this work 
was completed. The first named author was partially supported by \emph{SFB 1085 --- Higher Invariants} (University of Regensburg) funded by the DFG. The second named author was partially supported by the DFG priority programme \emph{SPP 2026 --- Geometry at infinity}.

\section{The cobordism category of a Waldhausen category}

\subsection{Edgewise subdivision} We recall Segal's edgewise sudivision $\Sd(X)$ of a simplicial set $X$ (see also \cite[Appendix 1]{Se}). Let 
$\Delta$ denote the usual category of finite ordinals $[n] = \{0 < 1 < \cdots < n\}$ and order-preserving maps. Let $(-)\op \colon \Delta \to \Delta$ be the standard involutive functor with $[n]\op = [n]$ and 
$$\alpha\op(k) = n - \alpha(m-k)$$
for each $\alpha \colon [m] \to [n]$ in $\Delta$. Then we define a functor 
$$\sd \colon \Delta \xrightarrow{(-)\op \times \mathrm{id}} \Delta \times \Delta \xrightarrow{- \ast -} \Delta$$
where $\ast$ denotes the ordinal sum. It will be convenient to represent
the well-ordered finite set $\sd[n] = [n]\op \ast [n] = [2n+1]$ as follows: 
$$\{\overline n < \overline{n-1} < \dots < \overline 0 < 0 < \dots < n\}.$$
Given a simplicial set $X \colon \Delta\op \to \set$, the edgewise subdivision is defined by 
$$\Sd(X): \Delta\op \xrightarrow{\sd\op} \Delta\op \xrightarrow{X} \set.$$
More specifically, we have that $(\Sd X)_n= X_{2n+1}$ and the simplicial operators of $\Sd(X)$ are determined by $X$ using 
the presentation of $[2n+1]$ shown above. Clearly this defines an endofunctor on the category of simplicial sets,
$$\Sd \colon \sset \to \sset,$$ 
which is induced by the endofunctor $\mu$. 

When $X$ is the nerve of a small category $\calc$, then $\Sd(X)$ is also the nerve of a category, namely, the \emph{twisted arrow category} $\mathrm{tw} \ \calc$ of $\calc$. Its objects are the arrows $c\to d$ in $\calc$, and a morphism from $c \to d$ to $c' \to d'$ is a commutative square in $\calc$
\[\xymatrix{
c \ar[r] & d \ar[d]\\
c' \ar[u] \ar[r] & d'.
}\]
The composition in $\mathrm{tw}\ \calc$ is given by concatenating such squares vertically, then composing vertical arrows and forgetting
the intermediate horizontal arrow.

\medskip

There is a canonical homeomorphism $|X| \cong |\Sd(X)|$ which, however, does not arise from a simplicial map, and will not be used in this paper. We will instead use 
a natural comparison map between simplicial sets called the last-vertex map 
$$L \colon \Sd(X) \to X$$ 
which is induced by the canonical inclusions $[n]\subset [n]\op \ast [n]$ for each $[n] \in \Delta$. In the case of (nerves of) categories, the last-vertex map corresponds to the functor $\mathrm{tw}\ \calc \to \calc$ which sends $(c\to d)$ to $d$ and $(c\to d) \to (c' \to d')$ to $d \to d'$. 

The last-vertex map $L \colon \Sd(X) \to X$ is known to be a natural weak equivalence of simplicial sets. (A proof of this claim can be given using the following standard method. 
Let $\mathcal E$ denote the class of simplicial sets $X$ such that $L\colon \Sd(X)\to X$ is a weak equivalence. We first note that $\Delta^n$, $n \geq 0$, is in $\mathcal E$, for 
$\Sd(\Delta^n)$ is the nerve of the category $\mathrm{tw}[n]$ and this has a terminal object. Moreover, the functor $\Sd$ commutes with colimits and it preserves monomorphisms. 
It follows that $\mathcal E$ is closed under pushouts along a monomorphism and under directed colimits of simplicial sets. Lastly, using the skeletal filtration of a simplicial set,
we conclude that the class $\mathcal E$ contains all simplicial sets.)

\subsection{Definition of the cobordism category}
For the definition and the basic properties of the cobordism category construction, it will suffice to work with an unpointed version of the notion of a Waldhausen category \cite{Wa}. 
By an \emph{unpointed Waldhausen category} we mean a small category $\calc$ equipped with subcategories of cofibrations $co\calc \subseteq \calc$ and of weak equivalences 
$w \calc \subseteq \calc$ that satisfy the following axioms (compare \cite{Wa}):
\begin{itemize}
 \item[(1)] Both $co \calc$ and $w \calc$ contain the isomorphisms in $\calc$.
 \item[(2)] For any morphism $C\to X$ and any cofibration $C\rightarrowtail D$, the pushout $X\cup_C D$ exists in $\calc$, and the induced map $X\to X\cup_C D$ is again a cofibration.
 \item[(3)] The weak equivalences satisfy the glueing lemma \cite[p.~326]{Wa}. 
\end{itemize}
This list contains exactly those of the axioms for a Waldhausen category that do not involve a zero object. An \emph{exact functor between unpointed Waldhausen categories} $F \colon \calc \to \calc'$ is 
a functor which preserves cofibrations, weak equivalences, and those pushout squares of the form described in axiom (2) above.

Clearly every Waldhausen category is an unpoined Waldhausen category. On the other hand, the categories of (unpointed) finite sets, and of (unbased) homotopy finite spaces are not 
pointed and hence do not underlie a Waldhausen category, but do admit the structure of an unpointed Waldhausen category. The latter example, and generalizations thereof, will be 
discussed in Section \ref{sec:models_for_A_theory}.

\medskip

Let $\calc$ be an unpointed Waldhausen category and let $\Cat$ denote the category of small categories. We define a simplicial category $\cob(\calc)_{\bullet} \colon \Delta\op \to \Cat$ as follows. The 
category $\cob(\calc)_n$ is the category of functors
\[
F \colon \mathrm{tw}[n]\to \calc, \quad (i \leq j)\mapsto F_{ij}
\]
such that
\begin{enumerate}
\item for each $i\leq j\leq k$, the map $F_{ij}\to F_{ik}$ is a cofibration, and 
\item for each $i\leq j\leq k\leq l$, the diagram in $\calc$
\[\xymatrix{
 & F_{il}\\
 F_{ik} \ar@{>->}[ru] && F_{jl} \ar[ul] \\
 & F_{jk} \ar@{>->}[ru] \ar[ul] &
}\] 
is a pushout square.
\end{enumerate}
The morphisms in $\cob(\calc)_n$ are the natural transformations between such functors. We denote $w \cob(\calc)_n \subset \cob(\calc)_n$ the subcategory of objectwise weak equivalences. 

Thus, $\cob(\calc)_0 = \calc$ and $w\cob(\calc)_0 = w \calc$. The category $\cob(\calc)_1$ (resp., $w\cob(\calc)_1$) consists of diagrams in $\calc$ of the form
\[
\xymatrix{
&F_{01} &\\
F_{00} \ar@{>->}[ur] & & F_{11} \ar[ul] \\
}
\]
and natural transformations (resp., natural weak equivalences) between them. We find it useful to think of such a diagram as some kind of (formal) cobordism from $F_{00}$ to $F_{11}$. In the next simplicial degree an object in $\cob(\calc)_2$ can be depicted as follows:
\[
\xymatrix{
&& F_{02} &&\\ 
& F_{01} \ar@{>->}[ur] && F_{12} \ar[ul] & \\
F_{00} \ar@{>->}[ur] & & F_{11} \ar@{>->}[ur] \ar[ul] && F_{22} \ar[ul] \\
}
\]
where the middle square is a pushout, and similarly for the higher simplicial degrees. In the language of cobordisms, the last diagram corresponds to the datum of two composable cobordisms
together with a choice of a composition. Continuing this analogy, an object in $\cob(\calc)_n$ corresponds to a string of $n$ many composable cobordisms, together with choices of compositions for all connected substrings. 

Both $\cob(\calc)_n$ and $w \cob(\calc)_n$ are natural in $[n]$ and define simplicial objects in the category of small categories $\Cat$. Moreover, $\cob(\calc)_{\bullet}$ also defines
a simplicial object in the category of unpointed Waldhausen categories and exact functors. (We will return to this point in Subsection \ref{deloopings}.)
The simplicial category $\cob(\calc)_{\bullet}$ (resp. $w\cob(\calc)_{\bullet}$) is a type of cospan category with 
restrictions, as imposed by condition (i).


\begin{defn}
The \emph{cobordism category} of the unpointed Waldhausen category $\calc$ is the simplicial space
$$ \cob(\calc, w\calc) \colon [n] \mapsto \big\vert N_{\bullet} w\cob(\calc)_n \big\vert.$$
The geometric realization of this simplicial space is called the \emph{classifying space} of the cobordism category and will be 
denoted by $\mathsf{B}\cob(\calc, w\calc)$. 
\end{defn}

The definition of the cobordism category is clearly functorial with respect to exact functors between (unpointed) Waldhausen categories. 

\begin{rem}
It is easy to see that $w\cob(\calc)$ defines a Segal object in $\Cat$, in the sense that the canonical restriction along the spine inclusion
$$w \cob(\calc)_n \xrightarrow{\simeq} \underbrace{w \cob(\calc)_1 \times_{w \cob(\calc)_0} \cdots \times_{w \cob(\calc)_0}
w\cob(\calc)_1}_{n}$$
is an equivalence of categories for each $n \geq 1$. Since geometric realization commutes with pullbacks, the cobordism category is also a 
Segal object in (a convenient model for) the category of spaces. This provides some justification of the term cobordism \emph{category}. 
However, we do not know at this level of generality whether $\cob(\calc, w\calc)$ is a Segal space in the usual sense -- that is, if additionally the iterated 
pullbacks
\[
 \vert w \cob(\calc)_1\vert \times_{\vert w \cob(\calc)_0\vert } \cdots \times_{\vert w \cob(\calc)_0\vert }
\vert w\cob(\calc)_1 \vert
\]
are also homotopy pullbacks. (This is true, for instance, if $w\calc$ is the class of isomorphisms in $\calc$.) 
\end{rem}


\subsection{Relative isomorphisms} Our main goal is to establish an equivalence between the cobordism category construction and Waldhausen's 
$S_{\bullet}$-construction in the case of a standard Waldhausen category. This will be done in Section \ref{section:comparison}. 
The main ingredient in the proof of this equivalence is the 
following useful method for showing that certain functors between cobordism categories induce homotopic maps between the classifying spaces.

Let $\calc$ be an unpointed Waldhausen category. We denote by $\cobpush(\calc)_{\bullet} \subset \cob(\calc)^{[1]}_\bullet$ the simplicial subcategory which is given 
in degree $n \geq 0$ by the full subcategory of $\cob(\calc)^{[1]}_n$ that is spanned by the functors 
$$F \colon \mathrm{tw}[n] \times [1] \to \calc, \quad (i \leq j, a)\mapsto F_{ij}^a$$
such that, for each $i\leq j\leq k$, the diagram in $\calc$
\[\xymatrix{
F_{ij}^0 \ar@{>->}[r] \ar[d] & F_{ik}^0 \ar[d] \\
F_{ij}^1 \ar@{>->}[r] & F_{ik}^1
}\]
is a pushout square. We also denote by $w \cobpush(\calc)_{\bullet} \subset \cobpush(\calc)_{\bullet}$ the subcategory of objectwise 
weak equivalences.  Note that there are morphisms of simplicial categories, 
$$
\xymatrix{
\cob(\calc)_{\bullet} \ar[r]^(.5){\delta_{\bullet}} & \cobpush(\calc)_{\bullet} \ar@<1ex>[r]^{s_{\bullet}} \ar@<-0.5ex>[r]_{t_{\bullet}} &\cob(\calc)_{\bullet} \\
}
$$
which are induced by the obvious maps $[0] \rightrightarrows [1] \rightarrow [0]$. These morphisms also restrict to the simplicial 
subcategories of weak equivalences.

\begin{lem} \label{useful-lem}
Let $\calc$ be an unpointed Waldhausen category. 
\begin{itemize}
\item[(a)] The inclusion of simplicial sets 
$$\delta_{\bullet} \colon  \mathrm{ob}(\cob(\calc)_{\bullet})\to \mathrm{ob}(\cobpush(\calc)_{\bullet})$$
is a weak equivalence (i.e., it induces a weak homotopy equivalence after geometric realization). 
\item[(b)] The inclusion of simplicial sets
\[
\delta_{\bullet} \colon N_n w\cob(\calc)_{\bullet}  \to N_n w\cobpush(\calc)_{\bullet}
\]
is a weak equivalence for each $n \geq 0$. As a consequence, there is a 
homotopy equivalence 
$$\delta \colon \mathsf{B} \cob(\calc, w\calc) \xrightarrow{\simeq} \big \vert N_{\bullet} w \cobpush(\calc)_{\bullet} \big\vert.$$ 
\item[(c)] The source and target projections $s_{\bullet}$ and $t_{\bullet}$ induce homotopic maps 
\[
s \simeq t \colon \big \vert N_{\bullet} w\cobpush(\calc)_{\bullet} \big \vert \rightarrow  \mathsf{B} \cob(\calc, w\calc).
 \] 
\end{itemize}
\end{lem}

\begin{proof}
Note that $s_{\bullet} \circ \delta_{\bullet} =\id = t_{\bullet} \circ \delta_{\bullet}$. Thus, for (a), it is enough to specify a simplicial homotopy 
\[H \colon \mathrm{ob}(\cobpush(\calc)_{\bullet}) \times \Delta^1 \to \mathrm{ob}(\cobpush(\calc)_{\bullet})\]
between $\delta_{\bullet} \circ s_{\bullet}$ and the identity map. This homotopy sends 
$(F, \alpha \colon [n]\to [1])$ to the composite
{\footnotesize
\[
\mathrm{tw}[n] \times [1] \xrightarrow{(\id, q) \times \id} \mathrm{tw}[n] \times [n]\op \times [1] \xrightarrow{\id \times \alpha\op \times \id} \mathrm{tw}[n]\times 
[1]\op\times [1]\xrightarrow{\id \times h} \mathrm{tw}[n] \times [1]\xrightarrow{F} \calc
\]
}where the first map is induced by the first-vertex map $q \colon \mathrm{tw}[n]\to [n]\op$, given by $(i \leq j) \mapsto i$, and the map $h \colon [1]\op\times [1]\to [1]$ is the standard (reverse) homotopy between $\id$ and the constant map at $0$, i.e., $h(1, i) = 0$ and $h(0, i) = i$. 

It follows from the definition that this is a simplicial map but we need to check that it really lands in $\cobpush(\calc)_{\bullet}$. Let $G$ denote the image of $(F,\alpha)$ under $H$. 
Then for fixed $\alpha$, we have that $G_{ij}^0=F_{ij}^0$ and $G_{ij}^1$ is either $F_{ij}^1$ or $F_{ij}^0$, depending on whether $\alpha(i)$ is $0$ or $1$, respectively. From this it follows easily that for each $i\leq j\leq k$ and any $a = 0,1$, the map
\[G_{ij}^a\rightarrowtail G_{ik}^a\]
is  a cofibration. Moreover, using our assumptions on $F$, it follows that for each $i\leq j\leq k \leq l$ and any $a=0,1$,  the diagram
\[\xymatrix{
G_{jk}^a \ar@{>->}[r] \ar[d] & G_{jl}^a \ar[d]\\
G_{ik}^a \ar@{>->}[r] & G_{il}^a
}\] 
is a pushout square. Similarly, for each $i\leq j\leq k$, the diagram
\[\xymatrix{
G_{ij}^0 \ar@{>->}[r] \ar[d] & G_{ik}^0 \ar[d] \\
G_{ij}^1 \ar@{>->}[r] & G_{ik}^1
}\]
is also a pushout square. Therefore the simplicial homotopy is well-defined and (a) follows. 

(b) We apply (a) to the unpointed Waldhausen category $w_n\calc$ whose objects are $n$-strings of weak equivalences in $\calc$, and whose morphisms are natural transformations 
of diagrams. This is a full subcategory of $\calc^{[n]}$ and is regarded as an unpointed Waldhausen category with the objectwise structure. Then the claim follows since
\[
\mathrm{ob}(\cob(w_n\calc)_\bullet) \cong N_n w\cob(\calc)_\bullet, \quad
\mathrm{ob} (\cobpush(w_n\calc)_\bullet) \cong N_n w\cobpush(\calc)_\bullet. 
\]
(c) is an immediate consequence of (the second part of) (b). 
\end{proof}

\begin{rem}
It is also possible to establish an additivity theorem for the cobordism category, analogous to Waldhausen's Additivity Theorem, 
by using similar arguments and following Waldhausen's proof in \cite{Wa}. 
\end{rem}  


\begin{defn}
Let $F, G \colon \calc \to \mathcal{D}$ be exact functors between unpointed Waldhausen categories. A natural transformation $\phi \colon F \to G$ is called a \emph{relative 
isomorphism} if for each cofibration $f \colon c \rightarrowtail c'$ in $\calc$, the diagram in $\mathcal{D}$
$$
\xymatrix{
F(c) \ar@{>->}[r]^{F(f)} \ar[d]_{\phi_c} & F(c') \ar[d]^{\phi_{c'}} \\
G(c) \ar@{>->}[r]_{G(f)} & G(c') 
}
$$
is a pushout square. 
\end{defn}

As an immediate consequence of Lemma \ref{useful-lem}, we have the following proposition. 

\begin{prop} \label{key-lem2}
Let $F, G \colon \calc \to \mathcal{D}$ be exact functors between unpointed Waldhausen categories and let $\phi \colon F \to G$ be a relative isomorphism. Then the induced 
maps 
$$\mathsf{B}\cob(F, wF) \simeq \mathsf{B}\cob(G, wG) \colon \mathsf{B}\cob(\calc, w\calc) \to \mathsf{B}\cob(\mathcal{D}, w\mathcal D)$$
are homotopic. 
\end{prop}
\begin{proof}
The natural transformation $\phi\colon \calc\to \mathcal D^{[1]}$ defines an exact functor of unpointed Waldhausen categories, where $\mathcal D^{[1]}$ is equipped with the objectwise structure. Then we obtain a simplicial functor 
\[\Phi_{\bullet} \colon = \cob(\phi)_\bullet\colon \cob(\calc)_\bullet\to \cob(\mathcal D^{[1]})_\bullet = \cob(\mathcal D)_\bullet^{[1]},\]
which by assumption lands in the simplicial subcategory $\cobpush(\mathcal D)_\bullet$. Moreover, $\Phi_{\bullet} \colon \cob(\calc)_{\bullet} \to \cobpush(\mathcal{D})_{\bullet}$ is 
such that $s_{\bullet} \circ \Phi_{\bullet}$ and $t_{\bullet} \circ \Phi_{\bullet}$ correspond to the simplicial functors induced by $F$ and $G$ respectively. Then 
the result follows from Lemma \ref{useful-lem}(c).
\end{proof}

\begin{exa}
Let $\calc$ be an unpointed Waldhausen category which has an initial object $\varnothing$ such that the unique morphism $\varnothing \to X$ is a cofibration for every object 
$X$ in $\calc$ (see also Section \ref{symmetric-section}). Then for any object $X$ in $\calc$, there is an exact ``shift'' functor 
$-\amalg X \colon \calc \to \calc$. The canonical natural transformation from the identity functor to this shift functor $(- \amalg X)$ is a relative isomorphism. By 
Proposition \ref{key-lem2}, it follows that the map induced on the cobordism category by the shift functor is homotopic to the identity.
\end{exa}

\subsection{A symmetric version} \label{symmetric-section}
The definition of the cobordism category construction is not symmetric under passing to the opposite category, because of the cofibration condition. 
In this subsection we present a symmetric variant of the cobordism category and show that it agrees with the original one in most cases of interest.

Let $\calc$ be an unpointed Waldhausen category which has an initial object $\varnothing$
such that the unique morphism $\varnothing \to X$ is a cofibration for each object $X \in \calc$ (so that, in particular, $\calc$ has finite coproducts). We will 
refer to such an unpointed Waldhausen category as an \emph{unpointed Waldhausen category with initial object}.

\medskip

For each $n \geq 0$, we consider the full subcategory 
$$\cobcof(\calc)_{n} \subset \cob(\calc)_{n}$$
which is spanned by the functors $F \colon \mathrm{tw}[n] \to \calc$ in $\cob(\calc)_n$ such that in addition
\begin{enumerate}
\item[(iii)] for each $0 \leq i < n$, the canonical map 
$$F_{ii} \sqcup F_{i+1, i+1} \to F_{i, i+1}$$ 
is a cofibration.
\end{enumerate}
Under this extra assumption, for each $0 \leq i \leq j \leq k \leq n$, the map $F_{jk} \to F_{ik}$ is a cofibration, too. Then $\cobcof(\calc)_{\bullet} \colon [n] \mapsto \cobcof(\calc)_n$ 
defines a semi-simplicial object in $\Cat$, i.e., a functor on the subcategory $\Delta_{<}^{\op} \subset \Delta^{\op}$ which consists of the injective maps.  (In the category--theoretic 
interpretation of $\cobcof(\calc)_\bullet$, the absence of degeneracies can be interpreted as the absence of identity morphisms.) We also have the corresponding semi-simplicial subcategory of weak equivalences 
$w \cobcof(\calc)_{\bullet} \subset \cobcof(\calc)_{\bullet}$.

\begin{rem}
The construction $\calc \mapsto \cobcof(\calc)_{\bullet}$ works also under the weaker assumption on $\calc$, namely, that a pushout of $D\leftarrow C\to X$ exists in $\calc$
when \emph{both} maps $C\to D$ and $C\to X$ are cofibrations.
\end{rem}

\begin{defn}
Let $\calc$ be an unpointed Waldhausen category with initial object. The \emph{symmetric cobordism category} of $\calc$ is the semi-simplicial space 
$$ \cobcof(\calc, w\calc) \colon [n] \mapsto \big\vert N_{\bullet} w\cobcof(\calc)_n \big\vert.$$
The geometric realization of this semi-simplicial space is called the \emph{classifying space} of the symmetric cobordism category and will 
be denoted by $\mathsf{B}\cobcof(\calc, w\calc)$.
\end{defn}

We compare this symmetric construction with the previous cobordism category construction in the case where $\calc$ has functorial factorizations. We recall that
$\calc$ is said to have functorial factorizations if every morphism in $\calc$ can be written functorially in the arrow category $\calc^{[1]}$ as the composition of a cofibration followed by a weak equivalence. For Waldhausen categories, this is equivalent to the existence of a cylinder functor satisfying 
the cylinder axiom in the sense of \cite[1.6]{Wa}.

The inclusion (of semi-simplicial categories) $w\cobcof(\calc)_{\bullet} \subset w\cob(\calc)_{\bullet}$  induces a map between the geometric
realizations of the corresponding semi-simplicial spaces. This can be composed with the canonical homotopy equivalence to the classifying space 
of $\cob(\calc, w\calc)_{\bullet}$, so we obtain a canonical map 
$$\mathsf{B} \cobcof(\calc, w\calc) \to \mathsf{B} \cob(\calc, w\calc).$$

\begin{prop} \label{refined2}
Let $\calc$ be an unpointed Waldhausen category with initial object. Suppose that $\calc$ has functorial factorizations. 
Then the inclusion map
$$\mathsf{B} \cobcof(\calc, w\calc) \to \mathsf{B} \cob(\calc, w\calc)$$
is a weak homotopy equivalence. 
\end{prop}
\begin{proof}
It suffices to prove that the inclusion of semi-simplicial categories induces a homotopy equivalence in each simplicial degree of the cobordism direction. Using the functorial factorizations, we can functorially replace 
each cospan 
\[
\xymatrix{
&F_{01} &\\
F_{00} \ar@{>->}[ur] & & F_{11} \ar[ul] \\
}
\]
by a weakly equivalent cospan 
\[
\xymatrix{
& \overline{F}_{01} &\\
F_{00} \ar@{>->}[ur] & & F_{11} \ar@{>->}[ul] \\
}
\]
where $F_{00} \sqcup F_{11} \rightarrowtail \overline{F}_{01} \xrightarrow{\sim} F_{01}$ is the functorial factorization.
Repeating this process and making choices of pushouts, we obtain 
a functor 
$$r \colon w \cob(\calc)_n \to w \cobcof(\calc)_n$$
such that both composites 
$$w \cob(\calc)_n \xrightarrow{r} w \cobcof(\calc)_n  \stackrel{\mathrm{incl}}{\hookrightarrow} w \cob(\calc)_n$$
$$w \cobcof(\calc)_n \stackrel{\mathrm{incl}}{\hookrightarrow} w \cob(\calc)_n \xrightarrow{r} 
w \cobcof(\calc)_n$$
are naturally weakly equivalent to the respective identity functors. Thus, the inclusion map is a homotopy equivalence for each $n \geq 0$, and the result follows. 
\end{proof}

\begin{rem}
There is an intermediate object between $\cobcof(\calc)_{\bullet}$ and $\cob(\calc)_{\bullet}$ where instead of (iii) above, 
we require only that the maps $F_{i+1,i+1} \to F_{i, i+1}$ are cofibrations. This defines a \emph{simplicial} subobject of $\cob(\calc)_{\bullet}$
and the associated classifying space has again the same homotopy type 
when $\calc$ has functorial factorizations. 
\end{rem}

\section{Comparison with the $S_{\bullet}$-construction} \label{section:comparison}

\subsection{Recollections} We recall Waldhausen's $S_{\bullet}$-construction from \cite{Wa}. Let $\calc$ be a (pointed) Waldhausen category
and let $\Ar[n] \colon = [n]^{[1]}$ denote the arrow category of $[n]$. The notation for an object $(i \leq j)$ of $\Ar[n]$ will be often 
abbreviated to $(ij)$.
The category $S_n \calc \subset \calc^{\Ar [n]}$ is the full subcategory spanned by the functors
$$F \colon \Ar[n] \to \calc, \quad (i \leq j) \mapsto F_{ij}$$
such that 
\begin{enumerate}
\item for each $i$, $F_{ii} = \ast$,
\item for each $i \leq j \leq k$, the map $F_{ij} \to F_{ik}$ is a 
cofibration,
\item for each $i \leq j \leq k \leq l$, the diagram in $\calc$
\[
\xymatrix{
F_{ik} \ar@{>->}[r] \ar[d] & F_{il} \ar[d] \\
F_{jk} \ar@{>->}[r] & F_{jl}
}
\]
is a pushout square. 
\end{enumerate}
The morphisms in $S_n \calc$ are given by natural transformations. 
The category $S_n \calc$ carries a natural Waldhausen category structure (see \cite{Wa}) where the weak equivalences are the objectwise weak equivalences between functors. We denote by $w S_n \calc \subset S_n \calc$ the subcategory of weak equivalences in $S_n \calc$. Both $S_n \calc$ and $w S_n \calc$ are natural in $[n]$ and they define simplicial objects in $\Cat$. Moreover, $S_{\bullet} \calc$ is a simplicial object in the category of Waldhausen categories and 
exact functors of Waldhausen categories.

By definition, we have $S_0 \calc = \{\ast\}$ and $S_1 \calc = \calc$. 
$S_2 \calc$ is the category of cofiber sequences in $\calc$. In the next 
simplicial degree, an object in $S_3 \calc$ can be depicted as a triangular staircase diagram of the form
\[
\xymatrix{
\ast \ar@{>->}[r] & F_{01} \ar@{>->}[r] \ar[d] & F_{02} \ar@{>->}[r] \ar[d] & F_{03} \ar[d] \\   
& \ast \ar@{>->}[r] & F_{12} \ar@{>->}[r] \ar[d] & F_{13} \ar[d] \\   
&& \ast \ar@{>->}[r] & F_{23} \ar[d] \\  
&&& \ast  
} 
\]
where each embedded square is a pushout. Following \cite{Wa}, we denote by $\F_n \calc \subset \calc^{[n]}$ the full subcategory spanned by filtered objects, i.e., the functors 
$$F \colon [n] \to \calc, \quad i \mapsto F_i$$
such that $F_i \to F_j$ is a cofibration for each $i \leq j$. Then the 
restriction functor $$S_n \calc \to \F_{n-1} \calc, \quad (F_{ij})_{0\leq i \leq j \leq n} \mapsto (F_{0j+1})_{0 \leq j \leq n-1}$$ is an equivalence of categories because $S_n \calc$ is obtained from $\F_{n-1} \calc $ simply by making choices of pushouts for 
each filtered object - however, $\F_{\bullet -1} \calc$ do not define a simplicial object. 

The algebraic $K$-theory $K(\calc)$ of $\calc$ \cite{Wa} is defined to be the loop space (based at the point represented by the zero object) of the geometric realization of the simplicial space $[n] \mapsto \big\vert N_{\bullet} w S_n \calc \big\vert$, i.e.,
$$K(\calc) \colon = \Omega \big\vert N_{\bullet} w S_{\bullet} \calc \big\vert.$$

\subsection{The comparison map}

Let $\calc$ be a Waldhausen category. We will compare the cobordism category of $\calc$ with the $S_{\bullet}$-construction of $\calc$. The comparison map is essentially defined by functors
\[\cob(\calc)_n \to \F_{n-1}(\calc), \quad F\mapsto \Bigg( \frac{F_{01}}{F_{00}}\rightarrowtail \dots \rightarrowtail \frac{F_{0n}}{F_{00}} \Bigg).\]
In order to promote this collection of functors to a simplicial functor mapping into the $S_{\bullet}$-construction, we must first modify our model for $\cob(\calc)_{\bullet}$ so that it includes choices of pushouts. For that purpose, we consider the full subcategory
\[\TAr[n]\subset \Ar[n]\]
spanned by objects $(i \leq j)$ where $\bar i \leq j$; here $\overline {(-)}$ denotes the order-reversing self-isomorphism of $[n]$ given by $i \mapsto n - i$. Thus, for example, the object $(1 \leq 2)$ of $\Ar[3]$ is an object of $\TAr[3]$ because $\bar 1 = 2 \leq 2$.  In more detail, the subposet 
$\TAr[3]$ is exactly the lower right half of the depicted $\Ar[3]$ 
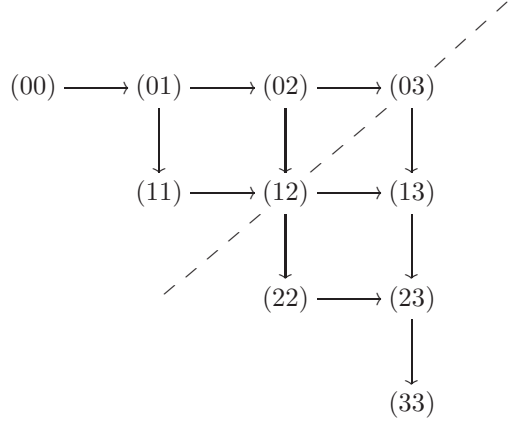
$$
\xymatrix{
&&&& \\
(00) \ar[r] & (01) \ar[r] \ar[d] & (02) \ar[r] \ar[d] & (03) \ar[d] \ar@{--}[ur] \\
& (11) \ar[r] & (12) \ar[r] \ar@{--}[ur] \ar@{--}[dl] \ar[d] & (13) \ar[d] \\
&  & (22) \ar[r] & (23) \ar[d] \\
&&& (33) 
}
$$ 
with respect to the indicated line that passes through $(03)$. Similarly, $\TAr[n] \subset \Ar[n]$ corresponds to its ``lower right half" with respect to the dichotomizing line that passes through $(0 \leq n)$. 

More specifically, we will consider the category $\TAr[2n+1] = \TAr([n]\op\ast[n])$ (using the canonical identification of $[2n+1]$ with $[n]\op \ast [n]$). An object of $\TAr[2n+1]$ is 
either of the form $(i \leq j)$, for any $n  < i \leq j \leq 2n + 1$, or of the form $(n - i \leq j)$ where $0 \leq i \leq n < j \leq n$. The first collection of objects defines a subposet isomorphic to $\Ar[n]$, while the second collection defines a subposet isomorphic to $\mathrm{tw}[n]$. Thus, we have full inclusions
\[\Ar[n]\subset \TAr([n]\op\ast[n])\supset\mathrm{tw}[n].\]
Here $\Ar[n]$ includes as objects of the first kind which form a smaller triangular staircase at the lower right part of the staircase diagram. The second inclusion of $\mathrm{tw}[n]$ includes the objects of the second kind. These subposets have empty intersection in $\TAr([n]\op\ast [n])$. In the example 
of $\TAr[3]$ above, the inclusion $\Ar[1] \subset \TAr[3]$ corresponds to the subposet 
$$
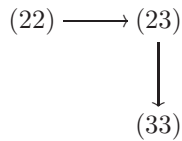
\xymatrix{
(22) \ar[r] & (23) \ar[d] & \\
& (33) 
}
$$
while the inclusion $\mathrm{tw}[1] \subset \TAr[3]$ corresponds to the subposet 
$$
\xymatrix{
& (13)  & \\
(03) \ar[ur] && (12). \ar[ul] \\
}
$$

We define the \emph{modified} cobordism category $\cobbig(\calc)_{\bullet}$ to be the simplicial category which 
in degree $n \geq 0$ is the category of functors
\[
F\colon \TAr([n]\op\ast[n])\to \calc
\]
satisfying the conditions in the $S_{\bullet}$-construction, namely:
\begin{enumerate}
\item for each $i$, $F_{ii}= \ast$,
\item for each $i\leq j\leq k$, the map $F_{ij}\to F_{ik}$ is a cofibration, 
\item for each $i\leq j\leq k\leq l$, the diagram in $\calc$
\[\xymatrix{
F_{ik} \ar@{>->}[r] \ar[d] & F_{il} \ar[d]\\
F_{jk} \ar@{>->}[r] & F_{jl}
}\]
is a pushout square.
\end{enumerate}
The morphisms in this category are given by natural transformations between such functors. We denote by $w \cobbig(\calc)_n \subset \cobbig(\calc)_n$ the subcategory 
of objectwise weak equivalences. Note that each $F \in \cobbig(\calc)_n$ is determined up to canonical isomorphism by its restriction along 
$$\mathrm{tw}[n] \subset \TAr([n]\op \ast [n])$$
since the rest of the diagram can be obtained by making choices 
of pushouts in $\calc$. ($F_{ij}$ is the cofiber of $F_{\bar{0}i} \rightarrowtail F_{\bar{0} j}$.) Thus, the restriction functors 
\[
\cobbig(\calc)_{\bullet} \to \cob(\calc)_{\bullet}, \quad w\cobbig(\calc)_{\bullet} \to w\cob(\calc)_{\bullet} 
\]
are degreewise equivalences of categories, and therefore they induce 
homotopy equivalences between the geometric realizations. In particular, we have a canonical homotopy equivalence
\begin{equation} \label{big-cob-small-cob}
\big \vert N_{\bullet} w\cobbig(\calc)_{\bullet} \big \vert \simeq \mathsf{B} \cob(\calc, w).
\end{equation}
Then we define a morphism between simplicial objects in $\Cat$ 
\[\tau_{\bullet} \colon  w \cobbig(\calc)_{\bullet} \to wS_{\bullet} \calc\]
induced by the restriction along the inclusion $\Ar[n] \subset \TAr([n]\op\ast[n])$. Passing to the geometric realizations, we obtain a natural (zigzag) comparison map of spaces

\smallskip

\[\tau \colon \mathsf{B} \cob(\calc, w\calc) \simeq \big \vert N_{\bullet} w\cobbig(\calc)_{\bullet} \big\vert 
\longrightarrow \big\vert N_{\bullet} wS_{\bullet} \calc \big\vert.\]

\smallskip 

\begin{thm} \label{main}
The comparison map $\tau$ is a weak homotopy equivalence.
\end{thm}
\begin{proof}
We may apply Segal's edgewise subdivision to the simplicial object $w S_{\bullet} \calc$ and obtain a new simplicial category $\Sd wS_{\bullet} \calc$ with $(\Sd w S_{\bullet} \calc)_n = w S_{2n+1} \calc$. The restriction along the inclusion $\TAr([n]\op\ast[n])\subset \Ar([n]\op\ast[n])$ induces a simplicial functor 
$$\alpha_{\bullet} \colon \Sd w S_{\bullet} \calc \to w\cobbig(\calc)_{\bullet}$$ 
such that the composite map of simplicial sets
\[
N_k \Sd w S_{\bullet} \calc  \xrightarrow{\alpha_{\bullet}}  N_k w \cobbig(\calc)_{\bullet}  \xrightarrow{\tau_{\bullet}} N_k w S_{\bullet} \calc,
\]
given by restriction along $\Ar[n]\subset \Ar([n]\op\ast[n])$, is the last-vertex map. This is a weak equivalence for each $k \geq 0$, and therefore so is also the induced composite map of spaces
\[
\big\vert  N_{\bullet} \Sd w S_{\bullet} \calc \big \vert \xrightarrow{\tau \circ \alpha} \big\vert N_{\bullet} w S_{\bullet} \calc\big \vert.
\]
Similarly, we may apply the edgewise subdivision to the modified cobordism category and consider the composite simplicial functor
\begin{equation} \label{composite-map}
\Sd w\cobbig(\calc)_{\bullet} \xrightarrow{\Sd\tau_{\bullet}} \Sd w S_{\bullet}(\calc) \xrightarrow{\alpha_{\bullet}} w\cobbig(\calc)_{\bullet} \simeq w\cob(\calc)_{\bullet}.
\end{equation}
This composition is \emph{not} the last-vertex map for $w\cobbig(\calc)_\bullet$. Indeed, after identifying the poset $[n]\op\ast[n]$ with its opposite, an $n$-simplex in the simplicial set of objects of $\Sd w\cobbig(\calc)_{\bullet}$ identifies with a diagram
\[\TAr\big( ([n]\op \ast[n])\ast ([n]\op\ast[n])\big)\to \calc.\]
Then $\alpha_{\bullet} \circ \Sd\tau_{\bullet}$ corresponds to the  restriction along the inclusion 
$$j_{34} \colon [n]\op\ast[n]\subset ([n]\op\ast[n])\ast ([n]\op\ast[n])$$
into the third and fourth factors. On the other hand, the last-vertex functor
\begin{equation} \label{last-vertex} 
L \colon \Sd w\cobbig(\calc)_{\bullet} \xrightarrow{L} w\cobbig(\calc)_{\bullet} \simeq w\cob(\calc)_{\bullet}
\end{equation}
corresponds to the restriction along the inclusion 
 $$j_{14} \colon [n]\op\ast[n]\subset ([n]\op\ast[n])\ast ([n]\op\ast[n])$$
into the first and fourth factors. However, the two inclusion functors are related by a (unique) natural transformation $j_{14} \Rightarrow j_{34}$, which induces a simplicial natural transformation 
\begin{equation} \label{nat-transformation} 
L \Rightarrow \alpha_{\bullet} \circ \Sd \tau_{\bullet} \colon \Sd w \cobbig(\calc)_{\bullet} \to w\cobbig(\calc)_{\bullet} \simeq w \cob(\calc)_{\bullet}.
\end{equation}
between the composite \eqref{composite-map} and the last-vertex functor \eqref{last-vertex}. 

This natural transformation \eqref{nat-transformation} has an extra property: given an object $G$ in the category $\Sd w\cobbig(\calc)_n$, the morphism 
$$F^0:=L(G)\to F^1:= (\alpha_{\bullet} \circ \Sd\tau_{\bullet})(G)$$ is such that the diagram
\[\xymatrix{
F^0_{ij} \ar@{>->}[r] \ar[d] & F^0_{ik} \ar[d]\\
F^1_{ij} \ar@{>->}[r] & F^1_{ik}
}\]
is a pushout square for each $0 \leq i\leq j\leq k \leq n$. (Indeed,  we have $F^0_{ij}=G_{\mu(i)\rho(j)}$ and $F^1_{ij}=G_{\nu(i)\rho(j)}$ where 
$\mu, \nu\colon [n]\op\to [n]\op\ast[n]\ast[n]\op\ast[n]$ are the two inclusions and $\rho\colon [n]\to [n]\op\ast[n]\ast[n]\op\ast[n]$ is the inclusion into the last factor.) 
This means that the natural transformation \eqref{nat-transformation} defines a simplicial functor 
$$H_{\bullet} \colon \Sd w \cobbig(\calc)_{\bullet} \to w \cobpush( \calc)_{\bullet}$$
such that the composition with the source and target projections, gives the last-vertex functor \eqref{last-vertex} and the simplicial functor \eqref{composite-map}, respectively. By Lemma \ref{useful-lem}, the source and target projections induce homotopic maps after geometric realization, therefore also the maps induced by \eqref{composite-map} and \eqref{last-vertex} must be homotopic. Since \eqref{last-vertex} induces the last-vertex map which is a weak homotopy equivalence, so is also the map induced by \eqref{composite-map} and the result follows.
\end{proof}

The classifying space $\mathsf{B}\cob(\calc, w\calc)$ is based at the zero object of the Waldhausen category $\calc$, denoted $*  \in w\calc = w \cob(\calc)_0$ and regarded as a $0$-simplex (`object') of the cobordism category 
$\cob(\calc, w\calc)$. This choice of basepoint is natural with respect to exact functors of Waldhausen categories. Moreover, the natural comparison map $\tau$ preserves the basepoint. Thus, passing to the loop spaces, we obtain the following corollary.  

\begin{cor} \label{main2}
The map $\Omega(\tau) \colon \Omega \ \mathsf{B} \cob(\calc, w\calc) \longrightarrow K(\calc)$
is a weak homotopy equivalence.
\end{cor}

Combined with Proposition \ref{refined2}, this also yields the following corollary. 

\begin{cor} \label{main3}
Let $\calc$ be a Waldhausen category with functorial factorizations. 
Then there is a natural (zigzag) weak homotopy equivalence 
$$\Omega \ \mathsf{B} \cobcof(\calc, w\calc) \simeq K(\calc).$$
\end{cor}

\subsection{Deloopings} \label{deloopings} Let $\calc$ be a Waldhausen category.
The application of the $S_{\bullet}$-construction can be iterated 
since $S_{\bullet}\calc$  defines a simplicial object in the category of Waldhausen categories and exact functors. It is well known that the iteration of the $S_{\bullet}$-construction produces canonical deloopings of $K(\calc)$, see \cite[1.5]{Wa}. 

The same is true for the cobordism construction $\calc \mapsto\cob(\calc)_{\bullet}$ applied to a Waldhausen category $\calc$. First, in order to see that the cobordism category construction can be iterated, it suffices to promote the simplicial category $\cob(\calc)_{\bullet}$ to a simplicial object 
in the category of Waldhausen categories. As subcategory of cofibrations $co \cob(\calc)_n \subset \cob(\calc)_n$, we consider the subcategory of
those natural transformations $\phi \colon (F_{\bullet\bullet}) \to (F'_{\bullet \bullet})$ such that for each $0 \leq i \leq n$, the morphism 
of filtered objects in $\F_{n-i+1} \calc$
$$
\xymatrix{
F_{ii} \ar@{>->}[r] \ar[d]_{\phi_{ii}} & F_{i, i+1} \ar@{>->}[r] \ar[d]_{\phi_{i,i+1}} & \cdots \ar@{>->}[r]& F_{in} \ar[d]_{\phi_{in}} \\
F'_{ii} \ar@{>->}[r]  & F'_{i, i+1} \ar@{>->}[r] & \cdots \ar@{>->}[r] & F'_{in} \\
}
$$
is a cofibration in $\F_{n-i+1} \calc$ (see \cite[1.1]{Wa}). This subcategory of cofibrations together with $w \cob(\calc)_n$ makes $\cob(\calc)_n$ 
into a Waldhausen category which is natural in $[n] \in \Delta^{\op}$. Thus, the cobordism category
construction can be iterated so that we obtain an $n$-fold simplicial category, for $n \geq 1$,
$$\calc \mapsto \cob^{(n)}(\calc)_{\bullet \cdots \bullet}.$$
We write $\mathsf{B}^{(n)} \cob(\calc, w)$ for the classifying space
of this multisimplicial category (i.e., the geometric realization of 
the associated $n$-fold simplicial space).

Let $\cob(\calc)_{1, \ast} \subset \cob(\calc)_1$ denote the full subcategory which consists of those diagrams $(F_{ij})$ such that $F_{00} = F_{11} = \ast$. Then we have $\cob(\calc)_{1, \ast} = \calc$. Each point in $w \cob(\calc)_{1, \ast}$ defines a loop in $\mathsf{B} \ \cob(\calc, w)$
and therefore we have a natural map 
\begin{equation} \label{group-comp-map} 
\big \vert w \calc \big\vert \to \Omega \ \mathsf{B} \cob(\calc, w\calc).
\end{equation}
Using the identification $\tau$, this agrees with the usual ``group completion" map $\big \vert w \calc \big \vert \to K(\calc)$. Moreover, using the naturality of this map \eqref{group-comp-map}, we also obtain natural maps 
\begin{equation} \label{spectrum}
\mathsf{B}^{(n-1)} \cob(\calc, w\calc) \to \Omega \ \mathsf{B}^{(n)} \cob(\calc, w)
\end{equation}
which make the sequence of spaces $\{\mathsf{B}^{(n)} \cob(\calc, w)\}_{n \geq 1}$ into a spectrum. As a consequence of Theorem \ref{main} and the naturality of $\tau$, we obtain natural (zigzag) weak homotopy equivalences
\begin{equation}
\Omega^n \mathsf{B}^{(n)} \cob(\calc, w) \simeq \Omega^n \big \vert N_{\bullet} w S^{(n)}_{\bullet} \calc \big \vert.
\end{equation}
These maps are also natural in $n \geq 1$, so they define a (zigzag) map of spectra. Using the fact that iterating the $S_{\bullet}$-construction defines canonical deloopings \cite[1.5]{Wa}, it follows that the maps \eqref{spectrum} are also weak homotopy equivalences. As a consequence, the spectrum $\{\mathsf{B}^{(n)} \cob(\calc, w\calc)\}_{n \geq 1}$ is an $\Omega$-spectrum.

\section{Example: Cobordism categories and $A$-theory}

\subsection{Models for $A$-theory} \label{sec:models_for_A_theory}
Let $X$ be a topological space and $\mathcal{R}^{(h)f}(X)$ the Waldhausen category of relative (homotopy) finite retractive space over $X$ (see \cite[2.1]{Wa}). It is well known that $\mathcal{R}^{(h)f}(X)$ has functorial factorizations given by a mapping cylinder construction. 

We now consider the (symmetric) cobordism categories associated to these two Waldhausen categories. Explicitly, in the case of $\Rfd(X)$, the symmetric cobordism category $\cobcof(\Rfd(X), w)$ is a semi-simplicial space such that:
\begin{enumerate}
\item the space of $0$-simplices is the moduli space of relative homotopy finite retractive spaces over $X$ (with respect to the class of homotopy equivalences),
\item the space of $1$-simplices is the moduli space of diagrams in $\Rfd(X)$ as follows:
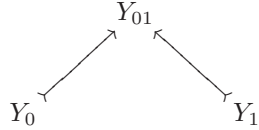
\[
\xymatrix{
&Y_{01} &\\
Y_0 \ar@{>->}[ur] & & Y_1 \ar@{>->}[ul] \\
}
\]
such that $Y_0 \cup_X Y_1 \to Y_{01}$ is a cofibration (with respect to the objectwise homotopy equivalences between such diagrams of spaces). Such a diagram may be regarded as a \emph{formal cobordism} between the retractive spaces $Y_0$ and $Y_1$. 
\item the space of $n$-simplices is the space of $n$-composable strings of $1$-simplices. 
\end{enumerate}
Theorems \ref{main} and \ref{main3} imply that there are natural (zigzag) homotopy equivalences in $X$
$$\Omega \ \mathsf{B} \cobcof(\Rfd(X), w) \simeq \Omega \ \mathsf{B} \cob(\Rfd(X), w) \simeq A(X) = K(\Rfd(X)).$$
There are similar homotopy equivalences in the case of $\mathcal{R}^{f}(X)$.

\smallskip 

More generally, for a fibration $p \colon E \to B$, we may consider 
the cobordism category $\cobcof(\Rfd(p), w)$ associated with the Waldhausen category of $\Rfd(p)$ (see \cite{RS1}). Then we obtain similarly natural (zigzag) 
homotopy equivalences in $p$,
\begin{equation} \label{A-theory1}
\Omega \ \mathsf{B} \cobcof(\Rfd(p), w) \simeq \Omega \ \mathsf{B} \cob(\Rfd(p), w) \simeq A(p) = K(\Rfd(p)),
\end{equation}
where $A(p)$ denotes the bivariant $A$-theory of $p$. 

\smallskip

Interestingly, the additional flexibility of working with unpointed Waldhausen categories furnishes yet another model for $A$-theory. Given a space $X$, let $\calc^{(h)f}(X)$ denote the 
category whose objects are pairs $(Y, u \colon Y \to X)$ where $Y$ is (homotopy equivalent to) a finite CW complex. A morphism $f \colon (Y, u) \to (Y', u')$ in $\calc^{hf}(X)$ is a map $f \colon Y \to Y'$ such that $u = u' f$. 
We say that this is a cofibration (resp. weak equivalence) if the underlying map $f$ is a Hurewicz cofibration (resp. homotopy equivalence). In the case of $\calc^f(X)$, the cofibrations are the inclusions of CW complexes. Note that the category $\calc^{(h)f}(X)$ has an initial 
object given by the pair $(\varnothing, \varnothing \to X)$. With this structure, $\calc^{(h)f}(X)$ becomes an unpointed Waldhausen category with initial object and functorial factorizations. 

The correspondence 
$X \mapsto \calc^{(h)f}(X)$ defines a functor from spaces to the category of unpointed Waldhausen categories which sends a map $f \colon X \to X'$ to the exact functor 
$$\calc^{(h)f}(X) \to \calc^{(h)f}(X'), \ (Z, Z \to X) \mapsto (Z, Z \to X \to X').$$
 The associated 
symmetric cobordism category $\cobcof(\calc^{hf}(X))_{\bullet}$ may be regarded as the category of formal cobordisms between homotopy finite spaces with a structure map to $X$.
By Proposition \ref{refined2}, we 
also have natural homotopy equivalences in $X$
\begin{equation} \mathsf{B}\cobcof(\calc^{(h)f}(X), w) \simeq \mathsf{B}\cob(\calc^{(h)f}(X), w).
\end{equation}

\smallskip

There are exact functors of unpointed Waldhausen categories relating $\calc^{hf}(X)$ and $\Rfd(X)$. First, there is an exact functor that adds a disjoint copy of $X$, 
$$\mathrm{J}_X \colon \calc^{hf}(X) \to \Rfd(X)$$
$$(Y, u \colon Y \to X) \mapsto (Y \sqcup X, X \subseteq Y \sqcup X \xrightarrow{u + \mathrm{id}_X} X).$$
Secondly, when $X$ is homotopy finite, there is an exact functor that forgets the section,
$$\mathrm{U}_X \colon \Rfd(X) \to \calc^{hf}(X)$$
$$(Z, X \rightarrowtail Z \xrightarrow{r} X) \mapsto (Z, r \colon Z \to X).$$
Note that the last functor does not preserve the initial object -- but it preserves pushouts. These functors restrict also to exact functors on $\calc^f(X)$ and $\mathcal{R}^f(X)$ 
respectively when $X$ is a finite CW complex.

\begin{prop} \label{A-theory-model}
Let $X$ be a homotopy finite space. Then the exact functors $\mathrm{U}_X$ and $\mathrm{J}_X$ induce inverse homotopy equivalences: 
$$\Omega \mathsf{B}\cob(\mathrm{J}_X, w \mathrm{J}_X) \colon \Omega  \mathsf{B}\cob(\calc^{hf}(X), w\calc^{hf}(X)) \to \Omega  \mathsf{B}\cob(\Rfd(X), w\Rfd(X))$$
$$\Omega \mathsf{B}\cob(\mathrm{U}_X, w\mathrm{U}_X) \colon \Omega  \mathsf{B}\cob(\Rfd(X), w\Rfd(X)) \to \Omega  \mathsf{B}\cob(\calc^{hf}(X), w\calc^{hf}(X)).$$
Moreover, the same is true for the restrictions of these functors to 
$\calc^f(X)$ and $\mathcal{R}^f(X)$ respectively, when $X$ is a finite CW complex.
\end{prop}
\begin{proof}
By Proposition \ref{key-lem2}, it suffices to show that both composite functors $\mathrm{U}_X \circ \mathrm{J}_X$ and $\mathrm{J}_X \circ \mathrm{U}_X$ are connected 
to the respective identity functors by relative isomorphisms. For the composite $\mathrm{U}_X \circ \mathrm{J}_X \colon \calc^{hf}(X) \to \calc^{hf}(X)$, 
$$(Y, u \colon Y \to X) \mapsto (Y \sqcup X, Y \sqcup X \xrightarrow{u + \mathrm{id}_X} X),$$
there is a relative isomorphism $\phi \colon \mathrm{Id} \to \mathrm{U}_X \circ \mathrm{J}_X$ where 
$$\phi_{(Y, u)} \colon Y \to Y \sqcup X$$
is the canonical inclusion. For the composite $\mathrm{J}_X \circ \mathrm{U}_X \colon \Rfd(X) \to \Rfd(X)$, given by
$$(Y, X \rightarrowtail Y \xrightarrow{r} X) \mapsto (Y \sqcup X, X \subseteq Y \sqcup X \xrightarrow{r + \mathrm{id}_X} X),$$ 
there is a relative isomorphism 
$\psi \colon \mathrm{J}_X \circ \mathrm{U}_X \to \mathrm{Id}$ where 
$$\psi_{(Y, X \stackrel{i}{\rightarrowtail} Y \xrightarrow{r} X)} \colon Y \sqcup X \xrightarrow{\mathrm{id}_Y + i} Y.$$ 
The same argument applies for the comparison of $\calc^f(X)$ and $\mathcal{R}^f(X)$.
\end{proof}

\begin{rem} \label{bivariant-simplified}
Similarly we can define an unpointed Waldhausen category $\calc^{hf}(p)$ for a fibration $p \colon E \to B$. The classifying space of its cobordism category is a 
model for $A(p)$ when $p$ has homotopy finite fibers. The proof is exactly the same. 
\end{rem}

\begin{prop} \label{finite-model}
The exact inclusion functor 
$\calc^f(X) \to \calc^{hf}(X)$ (of unpointed Waldhausen categories) induces a homotopy equivalence
$$\mathsf{B}\cob(\calc^{f}(X), w \calc^{f}(X)) \xrightarrow{\simeq} \mathsf{B}\cob(\calc^{hf}(X), w \calc^{hf}(X)).$$
\end{prop}
\begin{proof}
It is well known that the inclusion functor $\mathcal{R}^f(X) \to \Rfd(X)$ induces a $K$-equivalence  as a consequence of Waldhausen's Approximation 
Theorem (see \cite[Proposition 2.1.1]{Wa}). (Therefore the statement of the proposition for a finite CW complex $X$ follows directly from Proposition 
\ref{A-theory-model}.) The 
argument for the inclusion functor $\calc^f(X) \to \calc^{hf}(X)$ is similar so we only sketch the proof. For each $n \geq 0$, the functor (of unpointed Waldhausen categories) 
$$\cobcof(\calc^f(X))_n \to \cobcof(\calc^{hf}(X))_n$$ satisfies the conditions of the Approximation Theorem in \cite[pp. 35-37]{Wa}. For $n = 0$, the argument is similar to \cite[Proposition 2.1.1]{Wa}, and for $n > 0$, the proof of \cite[Lemma 1.6.6]{Wa} is easily adapted to this purpose. Then, following the proof of the Approximation Theorem in \cite{Wa}, or applying \cite[Lemma 7.6.7]{Cis}, we conclude that the functor 
$$w\cobcof(\calc^f(X))_n \to w\cobcof(\calc^{hf}(X))_n$$
induces a homotopy equivalence after passing to the classifying spaces -- as shown in the proof of \cite[Lemma 7.6.7]{Cis}, only the existence of an initial object, rather than a zero object, is required. 
The result then follows. 
\end{proof}

While $\mathrm{U}_X$ is well defined only when $X$ is homotopy finite and it is not natural in $X$, the functor $\mathrm{J}_X$ is a natural transformation 
defined for any $X$. Then we have the following result which generalizes Proposition \ref{A-theory-model} to an arbitrary $X$. 

\begin{thm} \label{model-A-theory2}
The natural map 
$$ \Omega \ \mathsf{B}\cob(\calc^{hf}(X), w\calc^{hf}(X)) \xrightarrow{\Omega \ \mathsf{B}\cob(\mathrm{J}_X, w\mathrm{J}_X)} \Omega \ \mathsf{B}\cob(\Rfd(X), w\Rfd(X)) \simeq A(X)$$
is a weak homotopy equivalence for any space $X$. 
\end{thm}
\begin{proof}
The functor $X \mapsto | w \cob(\calc^{hf}(X))_n |$ preserves homotopy equivalences because it sends the endpoint inclusions $i_0, i_1 \colon X \to X \times I$ to homotopic maps - each of these is homotopic to the map induced by the functor
$$(Z, u \colon Z \to X) \mapsto (Z \times I, u \times \mathrm{id} \colon Z \times I \to X \times I).$$
As a consequence, the functor 
\begin{equation} \label{any-X}
X \mapsto \mathsf{B} \cob(\calc^{hf}(X), w \calc^{hf}(X))
\end{equation} 
preserves homotopy equivalences, too. We claim that the functor \eqref{any-X} also preserves weak homotopy equivalences. Fix a functorial CW-approximation $g_X \colon X^c \xrightarrow{\sim} X$. 
Then there is a functor 
$$\Phi_n \colon w \cob(\calc^{hf}(X))_n \to w \cob(\calc^{hf}(X^c))_n$$
which is given by applying the functorial CW-approximation and replacing the maps in the cospans functorially by cofibrations as necessary. The composite functor
$$w \cob(\calc^{hf}(X))_n \xrightarrow{\Phi_n} w \cob(\calc^{hf}(X^c))_n \xrightarrow{g_X{}_*} w \cob(\calc^{hf}(X))_n$$
is weakly equivalent to the identity functor. The other composite functor
\begin{equation} \label{comp-functor} 
w \cob(\calc^{hf}(X^c))_n \xrightarrow{g_X{}_*} w \cob(\calc^{hf}(X))_n \xrightarrow{\Phi_n} w \cob(\calc^{hf}(X^c))_n 
\end{equation}
is described as follows: (i) first, it applies the CW-approximation functor (again), then (ii) it replaces the maps in the cospans by cofibrations as necessary, and lastly, 
(iii) it composes the induced structure map to $(X^c)^c$ with the map $g^c_{X} \colon (X^c)^c \to X^c$. This last map is homotopic to $g_{X^c} \colon (X^c)^c \to X^c$ and therefore 
after applying $| w \cob(\calc^{hf}(-))_n|$, these two maps $g^c_X$ and $g_{X^c}$ induce the same map up to homotopy. Thus, using in (iii) the map $g_{X^c}$ instead, we obtain 
a composite functor which is weakly equivalent to the identity functor and it induces a map homotopic to the one induced by \eqref{comp-functor}. It follows that  \eqref{any-X} sends $g_X$ to a homotopy equivalence. 
As a consequence, the functor \eqref{any-X} preserves weak homotopy equivalences. It is well known that the $A$-theory functor has this property too (see \cite[Proposition 2.1.7]{Wa}).

Then it suffices to prove that $\Omega \mathsf{B} \cob(\mathrm{J}_{X}, w \mathrm{J}_X)$ is a weak equivalence when $X$ is a CW complex. We may write $X$ as the (homotopy) filtered colimit of its 
finite subcomplexes. We note that the composite functor $\cob(\calc^f(-), w \calc^f(-))$, defined on objects by
$$X \mapsto \calc^f(X) \mapsto w \cob(\calc^f(X))_\bullet \mapsto \mathsf{B} \cob(\calc^f(X), w \calc^f(X)),$$ preserves this filtered colimit and therefore, after applying Proposition \ref{finite-model}, 
it follows that the functor \eqref{any-X} preserves this (homotopy) filtered colimit up to weak equivalence. Using the Waldhausen category $\mathcal{R}^f(X)$ for the definition of 
$A$-theory, the analogous argument shows that the same is true for the $A$-theory functor. Then the result follows by naturality and Proposition \ref{A-theory-model}.
\end{proof}

\subsection{The map from the cobordism category of manifolds.} Using the cobordism model for $A$-theory, we give a new description of the map from the cobordism category 
of \cite{GMTW, Ge} to $A$-theory, which was defined in \cite{BM} and studied further in \cite{RS1, RS2}. Here we will focus on the description of the map presented in \cite{RS2}.

\medskip

Let $\calc^{\partial}(\theta)$ denote the cobordism (non-unital) category of $\theta$-structured manifolds with boundary as defined in 
\cite[Section 2]{RS2}, where $\theta \colon X \to BO(d)$ is the fibration which determines the tangential structure. For fixed $\theta \colon X \to BO(d)$, 
there is a semi-simplicial map
$$N_{\bullet} \calc^{\partial}(\theta) \to \mathrm{Ob} \cobcof(\calc^{hf}(X))_{\bullet} \xrightarrow{\mathrm{J}_X} \mathrm{Ob} \cobcof(\Rfd(X))_{\bullet}$$
which sends an embedded $\theta$-structured cobordism $(W; M_0, M_1)$ to the cospan 
\[
\xymatrix{
& W \sqcup X &\\
M_0 \sqcup X \ar@{>->}[ur] & & M_1 \sqcup X \ar@{>->}[ul] \\
}
\]
of retractive spaces over $X$ (cf. \cite[Section 5.1]{RS1}, \cite[Section 5.1]{RS2}). Thus, this map corresponds essentially to an inclusion of cobordisms of compact smooth manifolds into (formal) cobordisms of homotopy finite spaces. We note that this semi-simplicial map preserves the basepoint that is defined by the initial object. After geometric 
realization, we obtain a map of spaces
$$\tau(\theta) \colon \Omega \mathsf{B}\calc^{\partial}(\theta) \to  \Omega\mathsf{B}\cobcof(\calc^{hf}(X), w\calc^{hf}(X)) \xrightarrow{\sim} \Omega\mathsf{B} \cobcof(\Rfd(X),w\Rfd(X)).$$

In this map, the topology on the cobordism category is not yet encoded. This can be rectified, just as in \cite{RS2}, by introducing an additional simplicial direction to obtain the \emph{simplicial thickening} $\calc^{\partial}(\theta)_{\bullet}$, a simplicial object with values in (non-unital) categories 
(see \cite[Section 2]{RS2}). Then the definition of the map above applies similarly in each simplicial 
degree and produces, as in \cite[Section 5.1]{RS2}, a simplicial morphism $\tau(\theta)_{\bullet}$ to the simplicial thickening (or \emph{thick model}) of
the symmetric cobordism category associated to $\Rfd(X)$, 
\[
 [n] \mapsto \Omega \mathsf B \cobcof\left(\Rfd \fibr{X\times \Delta^n}{\Delta^n}, w\Rfd\fibr{X\times \Delta^n}{\Delta^n}\right).
\]
The passage to the (geometric realization of the) simplicial thickening does not change the homotopy type of this cobordism category of retractive spaces. This can be seen either by following the arguments of \cite{RS2}, or by using the equivalence with bivariant $A$-theory and the analogous statement in this setup from \cite{BM}, \cite[Section 3.3]{RS1}. 

\smallskip 

Using the identification of Theorem \ref{main}, the (simplicial thickening of the) map $\tau(\theta)$ is precisely the map to $A$-theory as defined in \cite[Section 5.1]{RS2}.

\smallskip 

\noindent Note that the summand $X$ in the cospan above can be omitted by working instead with the simplicial unpointed Waldhausen category $\calc^{hf}(X)$ in order to model $A(X)$ 
(Theorem \ref{model-A-theory2}).

\end{document}